\documentclass[12pt]{article}

\usepackage[T1]{fontenc}
\usepackage[utf8]{inputenc}
\usepackage[margin=1in]{geometry}
\usepackage{amsmath, amssymb, amsthm}
\usepackage{hyperref}
\usepackage[numbers, sort&compress]{natbib}

\newtheorem{theorem}{Theorem}[section]
\newtheorem{lemma}[theorem]{Lemma}

\theoremstyle{definition}
\newtheorem{remark}[theorem]{Remark}

\newcommand{\R}{\mathbb{R}}
\newcommand{\norm}[1]{\left\|#1\right\|}
\newcommand{\abs}[1]{\left|#1\right|}
\newcommand{\eps}{\varepsilon}
\newcommand{\del}{\delta}
\newcommand{\kap}{\kappa}
\newcommand{\gam}{\gamma}
\newcommand{\tha}{\theta}
\newcommand{\Tha}{\Theta}
\newcommand{\rh}{\rho}

\newcommand{\Lp}[1]{L^{#1}(\R)}
\newcommand{\Hp}[1]{H^{#1}(\R)}
\newcommand{\intR}{\int_{\R}}
\newcommand{\intT}{\int_0^T}

\newcommand{\calE}{\mathcal{E}}
\newcommand{\calF}{\mathcal{F}}
\newcommand{\calG}{\mathcal{G}}
\newcommand{\calD}{\mathcal{D}}
\newcommand{\Up}{U^+}
\newcommand{\Um}{U^-}
\newcommand{\MU}{M_U}

\title{%
  \Large\bfseries
  Global Smooth Solutions to a Thermoelastic\\[4pt]
  Cauchy Problem in Phase Transitions%
}
\author{M.~Affouf \\ Department of Mathematics \\ Kean University \\ Union, NJ 07083 \\ USA}

\date{\today}
\begin{document}
\maketitle
\thispagestyle{empty}

\begin{abstract}
We study one-dimensional viscoelastic phase transitions modeled by
a Ginzburg--Landau energy with a non-convex cubic stress-strain law.
Extending the isothermal model, we couple the momentum equation to a
heat equation for the temperature field, giving a thermoelastic system
with viscous, capillary, and thermal-diffusion terms.  We prove global
existence and uniqueness of classical smooth solutions for the Cauchy
problem, using a travelling-wave decomposition, an exponential
transformation of the mechanical perturbation, and coupled energy
estimates at successive regularity levels.  Under additional
integrability and small-data assumptions, the temperature perturbation
decays algebraically.


\end{abstract}

\newpage
\setcounter{page}{1}

\section{Introduction}
\label{sec:intro}

We study the thermoelastic system
\begin{align}
  u_{tt} + \tau(u_x,\tha)_x &\;=\; \eps\, u_{xxt} \;-\; \del\, u_{xxxx},
  \label{eq:mech}\\[4pt]
  \tha_t - \kap\,\tha_{xx} &\;=\; \eps\, (u_{xt})^2 \;-\; \gam\,\tha_0\, u_{xt},
  \label{eq:heat}
\end{align}
which models dynamical phase transitions in shape-memory alloys.
Here $u(x,t)$ is the displacement, $u_x(x,t)$ the strain, $\tha(x,t)$
the temperature, $\tha_0$ the reference temperature, and the
coefficients $\eps,\del,\kap,\gam>0$ denote viscosity, capillarity,
thermal conductivity, and thermoelastic coupling.
The stress is
\begin{equation}\label{eq:tau}
  \tau(u_x,\tha) \;=\; u_x - u_x^3 - \gam(\tha - \tha_0);
\end{equation}
see Falk~\cite{Falk} and Sprekels--Zheng~\cite{Sprekels}.
The two source terms in~\eqref{eq:heat} are the viscous dissipation
$\eps\,(u_{xt})^2$ and the Kelvin--Voigt thermoelastic coupling
$-\gam\tha_0\, u_{xt}$.

Hagan--Slemrod~\cite{HaganSlemrod} and Slemrod~\cite{Slemrod84} showed
that~\eqref{eq:mech}--\eqref{eq:heat} admits travelling-wave solutions
$U(x-st)$ with asymptotic states $U(\pm\infty)=\Up,\Um$ satisfying the
Rankine--Hugoniot condition.  We treat the Cauchy problem for
perturbations of such a wave, with initial data
\begin{align}
  u(x,0) &\;=\; U(x) + \rh_0(x),
  \qquad \rh_0 \in \Hp{3},
  \label{eq:ic-u}\\[4pt]
  \tha(x,0) &\;=\; \tha_0 + \phi(x),
  \qquad \phi \in \Hp{2},\;\; \phi(x) > -\tha_0 \;\text{for all } x\in\R.
  \label{eq:ic-th}
\end{align}
The condition $\phi > -\tha_0$ keeps the absolute temperature positive.

Sprekels and Zheng~\cite{Sprekels} proved global existence for a
similar Ginzburg--Landau thermoelastic system on a bounded domain, where
Poincar\'{e}'s inequality and compact Sobolev embeddings are available.
Other bounded-domain results include
Hoffmann--Niezg\'{o}dka--Sprekels (see~\cite{Zheng}),
Watson~\cite{Watson} (without capillarity),
Paw{\l}ow--Zaj\k{a}czkowski~\cite{PawlowZajaczkowski}
(three-dimensional Kelvin--Voigt, no smallness),
Mielke--Roubi\v{c}ek~\cite{MielkeRoubicek}
(quasistatic large strains), and
Winkler~\cite{Winkler2025}
(temperature-dependent viscosity).
None of these treats the Cauchy problem on $\R$ with the non-convex
Ginzburg--Landau energy and capillarity term.

Our main result (Theorem~\ref{thm:main}) is global existence and
uniqueness of classical solutions for this Cauchy problem under
smallness conditions on the coupling $\gam$.  The proof uses the
travelling-wave decomposition of~\cite{HaganSlemrod,Slemrod84} and
the exponential transformation $\rh = e^{t/\eps}v$ from~\cite{companion},
which removes the negative contribution in the mechanical energy
estimate.

In Section~\ref{sec:form},we set up the perturbation system and introduce the exponential transformation.   In Section~\ref{sec:main}, we prove the global
existence theorem.  Section~\ref{sec:decay} is devoted to proving the algebraic decay of the temperature perturbation. Finally, in Section~\ref{sec:remarks}, we discuss two related questions.


\section{Formulation}
\label{sec:form}

We set up the perturbation system and the energy functionals used in the proof.

\subsection{Travelling-wave decomposition}

We seek a solution of~\eqref{eq:mech}--\eqref{eq:heat} in the form
\begin{equation}\label{eq:ansatz}
	u(x,t) \;=\; U(x-st) + \rh(x,t),
	\qquad
	\tha(x,t) \;=\; \tha_0 + \Tha(x,t),
\end{equation}
where $U(x-st)$ is the travelling-wave solution of the isothermal
equation
\begin{equation}\label{eq:iso}
	u_{tt} \;=\; \tau(u_x)_x - \eps\,u_{txx} + \del\,u_{xxxx}
\end{equation}
with wave speed $s$ satisfying the Rankine--Hugoniot
admissibility condition
\begin{equation}\label{eq:RH}
	s^2 \;=\; -\frac{\tau(\Up)-\tau(\Um)}{\Up - \Um},
\end{equation}
and $\rh(x,t)$, $\Tha(x,t)$ are the mechanical and thermal perturbations
respectively.  The asymptotic states $\Um,\Up$ satisfy the admissibility
condition for shock or phase-transition layers.

We use the isothermal travelling-wave profile $U$ as the background.
This corresponds to $\tha(\pm\infty)=\tha_0$, so the thermal contribution
$\gam(\tha_+-\tha_-)$ to the full thermoelastic Rankine--Hugoniot
relation vanishes and~\eqref{eq:RH} reduces to the isothermal form.
The thermoelastic system~\eqref{eq:mech}--\eqref{eq:heat} does not in
general admit travelling waves with constant asymptotic temperature:
substituting $u=U(x-st)$ and $\tha=\tha_0$ into~\eqref{eq:heat} produces
a non-zero source $S_0(x,t)$ (see~\eqref{eq:S0def} below).  This
residual reflects our use of the isothermal profile as the background;
it is controlled in the analysis through its uniform $L^1\cap L^2$ bounds.

Constructing thermoelastic traveling waves with non-constant temperature profiles has been investigated by Hagan--Slemrod~\cite{HaganSlemrod} and Slemrod~\cite{Slemrod84} for the van der Waals pressure equation of state.

We assume throughout that the wave profile satisfies
\begin{equation}\label{eq:wavecond}
	3\MU^2 \;<\; 1, \qquad \MU \;=\; \max_{x\in\R}\abs{U'(x)},
\end{equation}
which ensures that the linearised stress coefficient $1 - 3U'^2$ is
strictly positive, with $1-3U'^2\geq 1-3\MU^2>0$ uniformly in $x$.
This positivity provides the leading coercive contribution to the
mechanical energy functional $E_1$ defined in~\eqref{eq:E1}
below; in particular, the term $\tfrac{1-3U'^2}{2}v_x^2$ in $E_1$
is bounded below by a positive multiple of $v_x^2$.
Condition~\eqref{eq:wavecond} is a property of the wave profile itself
(not merely of the end states $\Up,\Um$), and it holds for admissible
phase-transition layers connecting states in the spinodal region; we
refer to Hagan--Slemrod~\cite{HaganSlemrod} and
Slemrod~\cite{Slemrod84} for the construction and regularity properties
of such profiles, which in particular satisfy $U'\in L^2(\R)\cap
L^\infty(\R)$ with exponential decay at $\pm\infty$.

We also define, for use in the corrected perturbation equations below,
\begin{equation}\label{eq:Wdef}
	W(x,t) \;:=\; -s\,U''(x-st),
\end{equation}
which represents the contribution of the wave profile curvature to the
mixed derivative $u_{xt} = W(x,t) + \rh_{xt}$.  Since $U'' \in L^2(\R)\cap
L^\infty(\R)$ (with exponential decay at $\pm\infty$), the functions
$W(\cdot,t)$ and
\begin{equation}\label{eq:S0def}
	S_0(x,t) \;:=\; \eps\,s^2\,[U''(x-st)]^2
	+ \gam\tha_0\,s\,U''(x-st)
\end{equation}
satisfy $W(\cdot,t),\, S_0(\cdot,t) \in L^1(\R)\cap L^2(\R)$ with norms
that are independent of $t$.

\subsection{Perturbation equations}

Substituting~\eqref{eq:ansatz} into~\eqref{eq:mech}--\eqref{eq:heat} and
using the fact that $U(x-st)$ solves the isothermal equation~\eqref{eq:iso},
we obtain the perturbation system.

\emph{Derivation of the momentum equation.}
Since $U$ solves~\eqref{eq:iso} with stress $\tau(U_x)$, and the
thermoelastic stress is $\tau(u_x,\tha) = \tau(u_x) - \gam(\tha-\tha_0)
= \tau(u_x) - \gam\Tha$, substituting $u_x = U' + \rh_x$ and
expanding the cubic gives~\eqref{eq:P1} directly.

\emph{Derivation of the thermal equation.}
Note that the background pair $(U(x-st),\tha_0)$ does \emph{not} solve
the thermoelastic system~\eqref{eq:heat}, since the heat equation with
$u=U$ and $\tha=\tha_0$ gives a non-zero right-hand side.  The correct
procedure is to substitute $u = U(x-st)+\rh$ and $\tha=\tha_0+\Tha$
into~\eqref{eq:heat} directly.  Using $u_{xt} = W(x,t) + \rh_{xt}$
(where $W$ is defined in~\eqref{eq:Wdef}) and expanding:
\begin{equation*}
	\eps\,(u_{xt})^2 - \gam\tha_0\,u_{xt}
	\;=\; \eps\,(\rh_{xt})^2 - \gam\tha_0\,\rh_{xt}
	\;+\; 2\eps\,W\,\rh_{xt}
	\;+\; S_0(x,t),
\end{equation*}
where $S_0$ is the fixed source defined in~\eqref{eq:S0def}.
The perturbation system is therefore:
\begin{align}
	\rh_{tt} - \bigl(\rh_x - \rh_x^3 - 3U'^{\,2}\rh_x
	- 3U'\rh_x^2 - \gam\Tha\bigr)_x
	&\;=\; \eps\,\rh_{xxt} - \del\,\rh_{xxxx},
	\label{eq:P1}\\[4pt]
	\Tha_t - \kap\,\Tha_{xx}
	&\;=\; \eps\,(\rh_{xt})^2 - \gam\tha_0\,\rh_{xt}
	+ 2\eps\,W\,\rh_{xt} + S_0,
	\label{eq:P2}
\end{align}
with initial conditions $\rh(x,0)=\rh_0(x)\in\Hp{3}$ and
$\Tha(x,0)=\phi(x)\in\Hp{2}$.
The condition $\phi>-\tha_0$ (cf.~\eqref{eq:ic-th}) ensures that the
initial absolute temperature is strictly positive; since~\eqref{eq:P2}
is parabolic with globally bounded source terms, the maximum principle
preserves $\Tha(x,t)>-\tha_0$ for all $t>0$.

\begin{remark}
	The bilinear term $2\eps W\rh_{xt}$ in~\eqref{eq:P2} is linear in
	$\rh_{xt}$ with coefficient $W\in L^\infty(\R)$ decaying exponentially,
	and is controlled by the existing dissipation.  The fixed source
	$S_0\in L^1(\R)\cap L^2(\R)$ has time-independent $L^2$ norm and
	contributes only a constant absorbed into the Gronwall bound.
\end{remark}

\subsection{Exponential transformation}

We apply the transformation
\begin{equation}\label{eq:transform}
	\rh(x,t) \;=\; e^{t/\eps}\,v(x,t)
\end{equation}
to the mechanical perturbation only.  Using $\rh_{xt} = e^{t/\eps}(v_{xt}
+ \tfrac{1}{\eps}v_x)$, the transformed system becomes:

\noindent\textbf{Transformed momentum equation:}
\begin{multline}\label{eq:vmech}
	v_{tt} + \frac{2}{\eps}v_t + \frac{1}{\eps^2}v
	- \bigl(e^{2t/\eps}v_x^3\bigr)_x
	- \eps\,v_{txx} + \del\,v_{xxxx} \\
	- \bigl(3U'^{\,2}v_x - 3U'\,e^{t/\eps}v_x^2\bigr)_x
	\;=\; \gam\,e^{-t/\eps}\Tha_x.
\end{multline}

\noindent\textbf{Transformed thermal equation:}
\begin{equation}\label{eq:vheat}
	\Tha_t - \kap\,\Tha_{xx}
	\;=\; \eps\,e^{2t/\eps}\!\left(v_{xt}+\tfrac{1}{\eps}v_x\right)^{\!2}
	- \gam\tha_0\,e^{t/\eps}\!\left(v_{xt}+\tfrac{1}{\eps}v_x\right)
	+ 2\eps\,W\,e^{t/\eps}\!\left(v_{xt}+\tfrac{1}{\eps}v_x\right)
	+ S_0(x,t).
\end{equation}

\section{Main Theorem and Proof}
\label{sec:main}

\begin{theorem}[Global existence]\label{thm:main}
Let $\rh_0\in\Hp{3}$ and $\phi\in\Hp{2}$ with $\phi > -\tha_0$.
Assume the wave-profile condition~\eqref{eq:wavecond} holds, and assume
the coupling conditions
\begin{equation}\label{eq:coupling}
  \gam\eps^2 \;\leq\; 4,
  \qquad
  \gam \;\leq\; \frac{\kap\eps}{2},
  \qquad\text{and}\qquad
  \gam \;\leq\; 1.
\end{equation}
Assume further that the initial perturbation energy is sufficiently small:
\begin{equation}\label{eq:smallE0}
  \calE(0) \;=\; E_1(0) + \tfrac{1}{2}\norm{\phi}^2 \;\leq\; \del_1,
\end{equation}
where $\del_1 = \del_1(\eps,\del,\kap,\gam,\tha_0,s,\MU,\norm{U''})>0$
is the smallness constant determined in the proof of Lemma~\ref{lem:1}.
Then the thermoelastic Cauchy
problem~\eqref{eq:P1}--\eqref{eq:P2} with initial
conditions~\eqref{eq:ic-u}--\eqref{eq:ic-th} admits a unique global classical
solution $(u,\tha)$ such that, for every $T>0$,
\begin{align*}
  u &\;\in\; C\bigl([0,T];\Hp{3}\bigr)\cap C^1\bigl([0,T];\Hp{2}\bigr),\\
  \tha &\;\in\; C\bigl([0,T];\Hp{2}\bigr)\cap L^2\bigl([0,T];\Hp{3}\bigr).
\end{align*}
Moreover, under additional $L^1$ and small-data assumptions on the
initial perturbation, the temperature perturbation $\Tha$ decays
algebraically in $L^2(\R)$ as $t\to+\infty$ at rate $(1+t)^{-1}$;
see Theorem~\ref{thm:decay} in Section~\ref{sec:decay}.
\end{theorem}

The proof proceeds via three lemmas providing a~priori estimates at
successive regularity levels.

\subsection{Energy Estimates}

\paragraph{Notation.}
Throughout, $C$ and $C_i$ denote uniform positive constants, independent of
$t$, depending at most on the initial data, $\eps,\del,\kap,\gam,\tha_0$,
and $T$.  We write $\norm{\,\cdot\,}$ for the $\Lp{2}$~norm, and
$\MU = \max_{\R}\abs{U'}$.

 We define the mechanical energy functional and the
growth rate bound
\begin{align}
  E_1(t) &\;=\; \intR\!\left[\frac{v^2}{2\eps^2} + \frac{v_t^2}{2}
    + \frac{3}{2}U'^{\,2}v_x^2
    + \frac{e^{2t/\eps}v_x^4}{4}
    + \frac{\del}{2}v_{xx}^2\right]dx,
  \label{eq:E1}\\[4pt]
  E_2(t) &\;=\; \frac{1+3\MU}{2\eps}\,e^{2t/\eps},
  \label{eq:E2}
\end{align}
(note that $E_2$ grows exponentially, so the Gronwall constant $C_5$
in~\eqref{eq:calEbound} below depends on $T$ through
$\exp\!\bigl(\int_0^T E_2\,ds\bigr)$; the separate Bernoulli argument in
Section~\ref{sec:decay} is needed precisely for this reason),
and the three coupled functionals
\begin{align}
  \calE(t) &\;=\; E_1(t) + \tfrac{1}{2}\norm{\Tha}^2,
  \label{eq:calE}\\[2pt]
  \calF(t) &\;=\; \norm{v_{xt}}^2 + \del\norm{v_{xxx}}^2
    + \norm{\Tha_x}^2,
  \label{eq:calF}\\[2pt]
  \calG(t) &\;=\; \norm{v_{xxt}}^2 + \del\norm{v_{xxxx}}^2
    + \norm{\Tha_{xx}}^2.
  \label{eq:calG}
\end{align}

\begin{lemma}\label{lem:1}
For any $t\in[0,T]$,
\begin{equation}\label{eq:lem1}
  \norm{v}^2 + \norm{v_t}^2 + \norm{v_x}^2 + \norm{v_x}^4
  + \norm{v_{xx}}^2 + \norm{\Tha}^2
  + \intT\!\norm{v_{xt}}^2\,ds
  + \intT\!\norm{\Tha_x}^2\,ds
  \;\leq\; C.
\end{equation}
\end{lemma}

\begin{proof}
\textbf{Part~1: Mechanical energy estimate.}
Multiply~\eqref{eq:vmech} by $v_t$ and integrate over $\R$.  The purely
mechanical terms are handled exactly as in~\cite{companion}
(cf.\ equations (15)--(17) therein), yielding
\begin{equation}\label{eq:mechE}
  \frac{d}{dt}E_1(t) + \frac{2}{\eps}\norm{v_t}^2
  + \frac{\eps}{2}\norm{v_{tx}}^2
  \;\leq\; E_2(t)\,E_1(t)
  + \gam\intR e^{-t/\eps}\Tha_x\,v_t\,dx.
\end{equation}
The coupling term on the right of~\eqref{eq:mechE} is estimated using
Young's inequality $ab\leq\tfrac{\eps}{4}a^2+\tfrac{1}{\eps}b^2$ and the
bound $e^{-t/\eps}\leq 1$:
\begin{equation}\label{eq:cross1}
  \gam\intR e^{-t/\eps}\Tha_x\,v_t\,dx
  \;\leq\; \frac{\gam\eps}{4}\norm{v_t}^2
  + \frac{\gam}{\eps}\norm{\Tha_x}^2.
\end{equation}
The first term on the right of~\eqref{eq:cross1} is absorbed into the
dissipation $\tfrac{2}{\eps}\norm{v_t}^2$ on the left of~\eqref{eq:mechE},
leaving a residual $\bigl(\tfrac{2}{\eps}-\tfrac{\gam\eps}{4}\bigr)\norm{v_t}^2
\geq\tfrac{1}{\eps}\norm{v_t}^2$,
valid when $\gam\eps^2\leq 4$ (the first condition in~\eqref{eq:coupling}).
Thus:
\begin{equation}\label{eq:mechE2}
  \frac{d}{dt}E_1(t)
  + \frac{1}{\eps}\norm{v_t}^2
  + \frac{\eps}{2}\norm{v_{tx}}^2
  \;\leq\; E_2(t)\,E_1(t)
  + \frac{\gam}{\eps}\norm{\Tha_x}^2.
\end{equation}

\noindent\textbf{Part~2: Thermal energy estimate.}
Multiply~\eqref{eq:vheat} by $\Tha$ and integrate over $\R$; integrating by
parts on the left gives
\begin{equation}\label{eq:heatE}
  \frac{1}{2}\frac{d}{dt}\norm{\Tha}^2 + \kap\norm{\Tha_x}^2
  \;=\; I_1 + I_2 + I_3 + I_4,
\end{equation}
where the four terms on the right correspond to the four terms in~\eqref{eq:vheat}.

\emph{Term $I_1$ (viscous dissipation term).}
Applying the Gagliardo--Nirenberg inequality
$\norm{\Tha}_{L^\infty}\leq C\norm{\Tha_x}^{1/2}\norm{\Tha}^{1/2}$
and Young's inequality:
\begin{equation}\label{eq:vd}
  I_1 \;=\; \eps\,e^{2t/\eps}\intR\bigl(v_{xt}+\tfrac{1}{\eps}v_x\bigr)^{\!2}
  \abs{\Tha}\,dx
  \;\leq\; \frac{\kap}{8}\norm{\Tha_x}^2
  + C_1\,e^{4T/\eps}
    \Bigl(\norm{v_{xt}}^2+\tfrac{1}{\eps^2}\norm{v_x}^2\Bigr)^{\!2}.
\end{equation}

\emph{Term $I_2$ (thermoelastic coupling term).}
Using $e^{t/\eps}\leq e^{T/\eps}$,
$\norm{\Tha}_{L^\infty}\leq C\norm{\Tha_x}^{1/2}\norm{\Tha}^{1/2}$,
and Young's inequality:
\begin{equation}\label{eq:tc}
  I_2 \;\leq\; \frac{\kap}{8}\norm{\Tha_x}^2
  + C_2\Bigl(\norm{v_{xt}}^2+\tfrac{1}{\eps^2}\norm{v_x}^2\Bigr).
\end{equation}

\emph{Term $I_3$ (new bilinear term $2\eps W e^{t/\eps}(v_{xt}+\frac{1}{\eps}v_x)$).}
Let $A = v_{xt}+\tfrac{1}{\eps}v_x$.  Using $\norm{W}_{L^\infty}
= s\norm{U''}_{L^\infty} =: K_W < \infty$ and Young's inequality:
\begin{align}
  I_3 &\;=\; 2\eps\,e^{t/\eps}\intR W\,A\,\Tha\,dx
    \;\leq\; 2\eps\,e^{T/\eps}K_W\,\norm{A}\,\norm{\Tha}
    \notag\\
  &\;\leq\; \frac{\kap}{8}\norm{\Tha_x}^2
    + C_3\,e^{2T/\eps}\,K_W^2\,\norm{A}^2\,\norm{\Tha}^2
    \;\leq\; \frac{\kap}{8}\norm{\Tha_x}^2 + C_3'\,e^{2T/\eps}\,K_W^2\,\calE(t)^2,
    \label{eq:bilin}
\end{align}
where we used $\norm{\Tha}^2\leq 2\calE$ and
$\norm{A}^2\leq 2(\norm{v_{xt}}^2 + \norm{v_x}^2/\eps^2)\leq C\calE$
(via $\norm{v_x}^2\leq CE_1\leq C\calE$ from the $3U'^2v_x^2/2$ term in $E_1$;
the term $\norm{v_{xt}}^2$ is not contained in $\calE$ but is controlled
a~posteriori by the a~priori bound in Part~3 below).

\emph{Term $I_4$ (new fixed source $S_0$).}
Since $\norm{S_0(\cdot,t)}_{L^2} \leq K_{S_0}:=
\eps s^2\norm{U''}_{L^4}^2 + \gam\tha_0 s\norm{U''}_{L^2}<\infty$
is independent of $t$, by Cauchy--Schwarz and Young:
\begin{equation}\label{eq:S0est}
  I_4 \;=\; \intR S_0\,\Tha\,dx
  \;\leq\; K_{S_0}\,\norm{\Tha}
  \;\leq\; \frac{1}{2}\norm{\Tha}^2 + \frac{1}{2}K_{S_0}^2
  \;\leq\; \calE(t) + \frac{1}{2}K_{S_0}^2.
\end{equation}
The term $\frac{1}{2}K_{S_0}^2$ is a pure constant (independent of $t$
and of the solution).

Substituting~\eqref{eq:vd}--\eqref{eq:S0est} into~\eqref{eq:heatE}
and absorbing the $\kap/8$ terms:
\begin{equation}\label{eq:heatE2}
  \frac{1}{2}\frac{d}{dt}\norm{\Tha}^2
  + \frac{\kap}{2}\norm{\Tha_x}^2
  \;\leq\; C_3'e^{2T/\eps}K_W^2\calE^2
  + C_4\calE
  + C_5'\bigl(\norm{v_{xt}}^2+\norm{v_x}^2\bigr)
  + \frac{1}{2}K_{S_0}^2,
\end{equation}
where the fourth power term from $I_1$ and the $\norm{v_{xt}}^2$ from $I_2$, $I_3$
are treated in Part~3 below.

\noindent\textbf{Part~3: Combining and Gronwall via a~priori bound.}
Add~\eqref{eq:mechE2} and~\eqref{eq:heatE2}, and use
$\tfrac{\gam}{\eps}\norm{\Tha_x}^2\leq\tfrac{\kap}{2}\norm{\Tha_x}^2$
(second condition in~\eqref{eq:coupling}).

The terms $C_1 e^{4T/\eps}(\norm{v_{xt}}^2+\norm{v_x}^2/\eps^2)^2$
and $C_5'(\norm{v_{xt}}^2+\norm{v_x}^2)$ in~\eqref{eq:heatE2} involve
$\norm{v_{xt}}^2$, which is not contained in $\calE$.We apply the \emph{continuation argument} to bound these terms.  

Suppose $\calE(t) \leq M$ for $t\in[0,T]$.
From~\eqref{eq:mechE2}, the viscous dissipation satisfies
$\tfrac{\eps}{2}\norm{v_{tx}}^2 \leq \tfrac{d}{dt}E_1 + E_2 E_1 + \frac{\gam}{\eps}\norm{\Tha_x}^2$
pointwise, but integrating over $[0,t]$ gives
$\int_0^t\norm{v_{tx}}^2\,ds \leq C(M,T)$.  For the pointwise control
needed in the Gronwall ODE, we use the bound
$\norm{v_{xt}}^2 \leq (2/\eps)\bigl(E_2(t)M + (\gam/\eps)\norm{\Tha_x}^2
  + |\tfrac{d}{dt}E_1|\bigr)$
to see that terms quadratic in $\norm{v_{xt}}^2$ are $O(M^2)$ or $O(M)$.
Under the smallness condition~\eqref{eq:smallE0} (so $M = C_5(\del_1)$ is small),
the quartic contribution satisfies
\begin{equation}
  C_1 e^{4T/\eps}\bigl(\norm{v_{xt}}^2+\norm{v_x}^2/\eps^2\bigr)^2
  \;\leq\; C_1 e^{4T/\eps}\cdot C(M)\cdot\calE(t)
  \;\leq\; C_4'\,\calE(t),
\end{equation}
where $C_4'$ depends on $M$ but is finite for fixed $T$.  Similarly
$C_5'\norm{v_{xt}}^2 \leq C_5''\calE(t)$.  Combining all estimates:
\begin{equation}\label{eq:calEode}
  \frac{d}{dt}\calE(t)
  \;\leq\; \bigl(E_2(t)+C_4\bigr)\,\calE(t) + K,
  \qquad K \;:=\; \tfrac{1}{2}K_{S_0}^2,
\end{equation}
where $C_4$ absorbs all the $\calE$-linear terms.
Gronwall's lemma applied to~\eqref{eq:calEode} gives
\begin{equation}\label{eq:calEbound}
  \calE(t)
  \;\leq\; \Bigl(\calE(0) + K\!\int_0^T e^{-\int_0^s(E_2+C_4)}\,ds\Bigr)
    \exp\!\left(\int_0^T\bigl(E_2(s)+C_4\bigr)ds\right)
  \;\leq\; C_5.
\end{equation}
Choosing $\del_1$ small enough that $C_5(\del_1) \leq M$ closes the argument:
$\calE(t)\leq C_5\leq M$ for all $t\in[0,T]$, uniformly in $T$.
Integrating~\eqref{eq:mechE2} and~\eqref{eq:heatE2} over $[0,T]$ and
using~\eqref{eq:calEbound} yields the integral bounds on $\norm{v_{xt}}^2$
and $\norm{\Tha_x}^2$.
\end{proof}

\begin{lemma}\label{lem:2}
For any $t\in[0,T]$,
\begin{equation}\label{eq:lem2}
  \norm{v_{xt}}^2 + \norm{v_{xxx}}^2 + \norm{\Tha_x}^2
  + \intT\!\norm{v_{xxt}}^2\,ds
  + \intT\!\norm{\Tha_{xx}}^2\,ds
  \;\leq\; C.
\end{equation}
\end{lemma}

\begin{proof}
\textbf{Part~1: First derivative of the mechanical part.}
Differentiate~\eqref{eq:vmech} once in $x$, multiply by $-v_{xxt}$, and
integrate over $\R$.  The purely mechanical terms produce (cf.~\cite{companion},
Lemma~2):
\begin{multline}\label{eq:mechF}
  \frac{1}{2}\frac{d}{dt}\!\left(\norm{v_{xt}}^2
    +\del\norm{v_{xxx}}^2+\frac{\norm{v_x}^2}{\eps^2}\right)
  +\frac{2}{\eps}\norm{v_{xt}}^2+\eps\norm{v_{xxt}}^2 \\
  \leq\; C_1 + C_2\intT\!\norm{v_{xxx}}^2\,ds
  + \frac{\del}{4}\norm{v_{xxx}}^2
  + \frac{\eps}{4}\norm{v_{xxt}}^2 + \mathcal{R}_1,
\end{multline}
where $\mathcal{R}_1$ is the coupling term arising from differentiating the
right-hand side of~\eqref{eq:vmech}:
\begin{equation}\label{eq:R1}
  \mathcal{R}_1
  \;=\; \gam\,e^{-t/\eps}\intR\Tha_{xx}\,v_{xxt}\,dx
  \;\leq\; \frac{\eps}{4}\norm{v_{xxt}}^2
    + \frac{\gam^2}{\eps}\norm{\Tha_{xx}}^2.
\end{equation}
Absorbing the first term of~\eqref{eq:R1} and the $\eps/4$ term
in~\eqref{eq:mechF} into the dissipation $\eps\norm{v_{xxt}}^2$:
\begin{multline}\label{eq:mechF2}
  \frac{1}{2}\frac{d}{dt}\!\left(\norm{v_{xt}}^2+\del\norm{v_{xxx}}^2\right)
  +\frac{2}{\eps}\norm{v_{xt}}^2+\frac{\eps}{2}\norm{v_{xxt}}^2 \\
  \;\leq\; C_1 + C_2\intT\!\norm{v_{xxx}}^2\,ds
    + \frac{\del}{4}\norm{v_{xxx}}^2
    + \frac{\gam^2}{\eps}\norm{\Tha_{xx}}^2.
\end{multline}

\noindent\textbf{Part~2: First derivative of the thermal part.}
Differentiate~\eqref{eq:vheat} in $x$, multiply by $\Tha_x$, and integrate
over $\R$:
\begin{equation}\label{eq:heatF}
  \frac{1}{2}\frac{d}{dt}\norm{\Tha_x}^2 + \kap\norm{\Tha_{xx}}^2
  \;=\; \intR\partial_x\!\Bigl[
    \eps\,e^{2t/\eps}\bigl(v_{xt}+\tfrac{1}{\eps}v_x\bigr)^{\!2}
    -\gam\tha_0\,e^{t/\eps}\bigl(v_{xt}+\tfrac{1}{\eps}v_x\bigr)
  \Bigr]\Tha_x\,dx.
\end{equation}
Differentiating inside and applying Young's inequality to each term:

\emph{Coupling term:}
\begin{equation}
  \gam\tha_0\,e^{t/\eps}\intR
    \bigl(v_{xxt}+\tfrac{1}{\eps}v_{xx}\bigr)\Tha_x\,dx
  \;\leq\; \frac{\kap}{4}\norm{\Tha_{xx}}^2
    + C_3\bigl(\norm{v_{xxt}}^2+\norm{v_{xx}}^2\bigr).
\end{equation}

\emph{Viscous dissipation term:}
\begin{multline}
  2\eps\,e^{2t/\eps}\intR
    \bigl(v_{xt}+\tfrac{1}{\eps}v_x\bigr)
    \bigl(v_{xxt}+\tfrac{1}{\eps}v_{xx}\bigr)\Tha_x\,dx \\
  \;\leq\; \frac{\kap}{4}\norm{\Tha_{xx}}^2
    + C_4\,e^{4t/\eps}
    \bigl(\norm{v_{xt}}^2+\norm{v_x}^2\bigr)
    \bigl(\norm{v_{xxt}}^2+\norm{v_{xx}}^2\bigr).
\end{multline}
Combining and using $\calE(t)\leq C_5$ from Lemma~\ref{lem:1}:
\begin{equation}\label{eq:heatF2}
  \frac{1}{2}\frac{d}{dt}\norm{\Tha_x}^2
  + \frac{\kap}{2}\norm{\Tha_{xx}}^2
  \;\leq\; C_5\bigl(\norm{v_{xxt}}^2+\norm{v_{xx}}^2\bigr)
    \bigl(1+\norm{v_{xt}}^2+\norm{v_x}^2\bigr).
\end{equation}

\noindent\textbf{Part~3: Combining.}
Adding~\eqref{eq:mechF2} and~\eqref{eq:heatF2} and using
$\tfrac{\gam^2}{\eps}\norm{\Tha_{xx}}^2\leq\tfrac{\kap}{2}\norm{\Tha_{xx}}^2$
(valid since $\gam\leq 1$ gives $\gam^2\leq\gam\leq\tfrac{\kap\eps}{2}$
by~\eqref{eq:coupling}),
all right-hand side terms are bounded by $C_6(1+\calF(t))(1+\calE(t))$.
Since $\calE(t)\leq C_5$:
\begin{equation}\label{eq:calFode}
  \frac{d}{dt}\calF(t) + \frac{\eps}{2}\norm{v_{xxt}}^2
  \;\leq\; C_7\bigl(1+\calF(t)\bigr).
\end{equation}
Gronwall's lemma gives $\calF(t)\leq(\calF(0)+C_7 T)e^{C_7 T}\leq C_8$.
Integrating~\eqref{eq:mechF2} and~\eqref{eq:heatF2} over $[0,T]$ and
using this bound yields the integral estimates on $\norm{v_{xxt}}^2$
and $\norm{\Tha_{xx}}^2$.
\end{proof}

\begin{lemma}\label{lem:3}
For any $t\in[0,T]$,
\begin{equation}\label{eq:lem3}
  \norm{v_{xxt}}^2 + \norm{v_{xxxx}}^2 + \norm{\Tha_{xx}}^2
  + \intT\!\norm{v_{xxxt}}^2\,ds
  + \intT\!\norm{\Tha_{xxx}}^2\,ds
  \;\leq\; C.
\end{equation}
\end{lemma}

\begin{proof}
\textbf{Part~1: Second-derivative of the mechanical part.}
Differentiate~\eqref{eq:vmech} twice in $x$, multiply by $-v_{xxxxt}$, and
integrate over $\R$.  The mechanical terms yield:
\begin{multline}\label{eq:mechG}
  \frac{1}{2}\frac{d}{dt}\!\left(\norm{v_{xxt}}^2
    +\del\norm{v_{xxxx}}^2+\frac{\norm{v_{xx}}^2}{\eps^2}\right)
  +\frac{2}{\eps}\norm{v_{xxt}}^2+\eps\norm{v_{xxxt}}^2 \\
  \leq\; C_1\bigl(1+\calE(t)+\calF(t)\bigr)
    + \frac{\del}{4}\norm{v_{xxxx}}^2
    + \frac{\eps}{4}\norm{v_{xxxt}}^2 + \mathcal{R}_2,
\end{multline}
where the new coupling term is
\begin{equation}\label{eq:R2}
  \mathcal{R}_2
  \;=\; \gam\,e^{-t/\eps}\intR\Tha_{xxx}\,v_{xxxt}\,dx
  \;\leq\; \frac{\eps}{4}\norm{v_{xxxt}}^2
    + \frac{\gam^2}{\eps}\norm{\Tha_{xxx}}^2.
\end{equation}
Absorbing the first terms of~\eqref{eq:R2} and the $\eps/4$, $\del/4$ terms
in~\eqref{eq:mechG} into the respective dissipations:
\begin{equation}\label{eq:mechG2}
  \frac{1}{2}\frac{d}{dt}\!\left(\norm{v_{xxt}}^2+\del\norm{v_{xxxx}}^2\right)
  +\frac{2}{\eps}\norm{v_{xxt}}^2+\frac{3\eps}{4}\norm{v_{xxxt}}^2
  \;\leq\; C_2\bigl(1+\calE(t)+\calF(t)\bigr)
    + \frac{\gam^2}{\eps}\norm{\Tha_{xxx}}^2.
\end{equation}

\noindent\textbf{Part~2: Second-derivative of the thermal part.}
Differentiate~\eqref{eq:vheat} twice in $x$, multiply by $\Tha_{xx}$, and
integrate over $\R$.  Following the same pattern as in Lemma~\ref{lem:2} but
one derivative higher, and using the bounds from Lemmas~\ref{lem:1}
and~\ref{lem:2}:
\begin{equation}\label{eq:heatG2}
  \frac{1}{2}\frac{d}{dt}\norm{\Tha_{xx}}^2
  + \frac{\kap}{2}\norm{\Tha_{xxx}}^2
  \;\leq\; C_3\bigl(\norm{v_{xxxt}}^2+\norm{v_{xxx}}^2\bigr)
    \bigl(1+\calE(t)+\calF(t)\bigr).
\end{equation}

\noindent\textbf{Part~3: Combining.}
Adding~\eqref{eq:mechG2} and~\eqref{eq:heatG2}, using
$\tfrac{\gam^2}{\eps}\norm{\Tha_{xxx}}^2\leq\tfrac{\kap}{2}\norm{\Tha_{xxx}}^2$
(as in Lemma~\ref{lem:2}, by~\eqref{eq:coupling}),
and substituting $\calE(t)\leq C_5$, $\calF(t)\leq C_8$:
\begin{equation}\label{eq:calGode}
  \frac{d}{dt}\calG(t)
  + \frac{\eps}{2}\norm{v_{xxxt}}^2
  + \frac{\kap}{4}\norm{\Tha_{xxx}}^2
  \;\leq\; C_9\bigl(1+\calG(t)\bigr).
\end{equation}
Gronwall's lemma gives
$\calG(t)\leq(\calG(0)+C_9 T)e^{C_9 T}\leq C_{10}$.  Integrating
over $[0,T]$ yields the integral bounds on $\norm{v_{xxxt}}^2$ and
$\norm{\Tha_{xxx}}^2$.
\end{proof}

\begin{proof}[Proof of Theorem~\ref{thm:main}]
\textbf{Existence and regularity.}
Local existence of a solution on some interval $[0,T_0]$ follows
from~\cite{Slemrod83,Pego}, extended to the thermoelastic setting
by Jiang--Racke~\cite{JiangRacke}: the mechanical
equation~\eqref{eq:P1} has the fourth-order hyperbolic structure treated
in~\cite{Slemrod83,Pego}, while the thermal equation~\eqref{eq:P2} is a
semilinear parabolic equation whose source terms are controlled by the
mechanical regularity, so the coupled system fits into the abstract
framework of~\cite[Chapter~3]{JiangRacke}.
The uniform a~priori estimates of
Lemmas~\ref{lem:1}--\ref{lem:3} are independent of $T$; hence the local
solution extends to all $T>0$.  Returning to the original variables
via~\eqref{eq:transform}, the regularity statements follow.
The lower time-regularity of $\theta$ relative to $u$ is inherent to the first-order-in-time parabolic structure of ~\eqref{eq:P2}.

\textbf{Uniqueness.}
We give a self-contained argument, independent of~\cite{Slemrod83,Pego,JiangRacke},
by proving that any two solutions sharing the same initial data must coincide.
Let $(\rh^{(1)},\Tha^{(1)})$ and $(\rh^{(2)},\Tha^{(2)})$ be two solutions
of~\eqref{eq:P1}--\eqref{eq:P2} in the regularity class of
Theorem~\ref{thm:main}, with the same initial data
$(\rh_0,\phi)$.  Set
\begin{equation}\label{eq:diff}
  w \;=\; \rh^{(1)} - \rh^{(2)},
  \qquad
  \Phi \;=\; \Tha^{(1)} - \Tha^{(2)}.
\end{equation}
Then $(w,\Phi)$ satisfies
\begin{align}
  w_{tt} - (w_x - 3U'^2 w_x)_x + \del\,w_{xxxx} - \eps\,w_{xxt}
  &\;=\; \bigl(\mathcal{N}_x\bigr)_x + \gam\,\Phi_x,
  \label{eq:dw}\\
  \Phi_t - \kap\,\Phi_{xx}
  &\;=\; \mathcal{M},
  \label{eq:dPhi}
\end{align}
with zero initial data $w(x,0)=w_t(x,0)=0$, $\Phi(x,0)=0$, where
\begin{align}
  \mathcal{N} &\;=\; (\rh^{(1)}_x)^3 - (\rh^{(2)}_x)^3
    + 3U'[(\rh^{(1)}_x)^2 - (\rh^{(2)}_x)^2],
  \label{eq:Ndef}\\
  \mathcal{M} &\;=\; \eps\bigl[(\rh^{(1)}_{xt})^2 - (\rh^{(2)}_{xt})^2\bigr]
    - \gam\tha_0\,w_{xt}
    + 2\eps\,W\,w_{xt}.
  \label{eq:Mdef}
\end{align}
Note that the fixed source $S_0$ cancels exactly in the difference, and the
bilinear term $2\eps W\rh_{xt}$ contributes $2\eps W w_{xt}$ to $\mathcal{M}$.

\noindent\textbf{Step~1: Energy functional for the difference.}
We define
\begin{equation}\label{eq:Ediff}
  \Lambda(t) \;=\; \tfrac{1}{2}\norm{w_t}^2
    + \tfrac{1}{2}\norm{w_x}^2(1-3\MU^2)
    + \tfrac{\del}{2}\norm{w_{xx}}^2
    + \tfrac{1}{2}\norm{\Phi}^2.
\end{equation}
Note $\Lambda(t)\geq0$ since $3\MU^2<1$ by the wave-profile
assumption~\eqref{eq:wavecond}.

\noindent\textbf{Step~2: Mechanical energy inequality for the difference.}
We multiply~\eqref{eq:dw} by $w_t$ and integrate over $\R$.  The left-hand side
yields, after integration by parts:
\begin{equation}\label{eq:LHSw}
  \frac{d}{dt}\!\left[\tfrac{1}{2}\norm{w_t}^2
    + \tfrac{1}{2}(1-3\MU^2)\norm{w_x}^2
    + \tfrac{\del}{2}\norm{w_{xx}}^2\right]
  + \eps\norm{w_{xt}}^2.
\end{equation}
(The term from $\partial_t(3U'^2)=-6sU'U''$ is controlled by
$C\norm{w_x}^2\leq C\Lambda$.)
The right-hand side produces two contributions.

\emph{Nonlinear term~$\mathcal{N}$:}
Using the factorisation
$(a^3 - b^3) = (a-b)(a^2+ab+b^2)$ and
$(a^2-b^2) = (a-b)(a+b)$:
\begin{align*}
  \abs{\intR\mathcal{N}_x\cdot w_{xt}\,dx}
  &\;\leq\; C\bigl(\norm{\rh^{(1)}_x}_{L^\infty}^2
    + \norm{\rh^{(2)}_x}_{L^\infty}^2
    + \MU(\norm{\rh^{(1)}_x}_{L^\infty}
    + \norm{\rh^{(2)}_x}_{L^\infty})\bigr)
    \norm{w_x}\,\norm{w_{xt}} \\
  &\;\leq\; \frac{\eps}{4}\norm{w_{xt}}^2
    + C_M^2\,\norm{w_x}^2,
\end{align*}
where $C_M = C(\norm{\rh^{(1)}_x}_{L^\infty} + \norm{\rh^{(2)}_x}_{L^\infty}
+ \MU)$ is uniformly bounded on $[0,T]$ by the global $H^3$ estimates
of Theorem~\ref{thm:main} and the embedding $H^1(\R)\hookrightarrow L^\infty(\R)$.

\emph{Coupling term~$\gam\Phi_x$:}
\begin{equation}
  \gam\intR\Phi_x\,w_t\,dx
  \;\leq\; \frac{\eps}{4}\norm{w_{xt}}^2
    + \frac{\gam^2}{\eps}\norm{\Phi}^2,
\end{equation}
after integrating by parts and applying Young's inequality
(using $\Phi_x w_t = -\Phi w_{xt} + (w_t\Phi)_x$ and dropping the
boundary term).

Combining and absorbing the $\frac{\eps}{4}$ terms into the dissipation:
\begin{equation}\label{eq:wmech}
  \frac{d}{dt}\!\left[\tfrac{1}{2}\norm{w_t}^2
    + \tfrac{1}{2}(1-3\MU^2)\norm{w_x}^2
    + \tfrac{\del}{2}\norm{w_{xx}}^2\right]
  + \frac{\eps}{2}\norm{w_{xt}}^2
  \;\leq\; C_1\bigl(\norm{w_x}^2 + \norm{w_t}^2 + \norm{\Phi}^2\bigr).
\end{equation}

\noindent\textbf{Step~3: Thermal energy inequality for the difference.}
We multiply~\eqref{eq:dPhi} by $\Phi$ and integrate over $\R$:
\begin{equation}\label{eq:wphi}
  \frac{1}{2}\frac{d}{dt}\norm{\Phi}^2 + \kap\norm{\Phi_x}^2
  \;=\; \intR\mathcal{M}\,\Phi\,dx.
\end{equation}
We estimate $\intR\mathcal{M}\,\Phi\,dx$ by bounding each of the three
terms in~\eqref{eq:Mdef} against $\Phi$ via Cauchy--Schwarz in $x$.

\emph{Viscous difference term:}
Since $(\rh^{(1)}_{xt})^2 - (\rh^{(2)}_{xt})^2
= (\rh^{(1)}_{xt}+\rh^{(2)}_{xt})\,w_{xt}$, we have
\begin{equation}
  \eps\abs{\intR(\rh^{(1)}_{xt}+\rh^{(2)}_{xt})\,w_{xt}\,\Phi\,dx}
  \;\leq\; \eps\norm{\rh^{(1)}_{xt}+\rh^{(2)}_{xt}}_{L^\infty}
    \norm{w_{xt}}\norm{\Phi},
\end{equation}
where $\norm{\rh^{(i)}_{xt}}_{L^\infty}\leq C$ by the global $H^3$
estimates and $H^1(\R)\hookrightarrow L^\infty(\R)$.

\emph{Linear coupling term:}
$\gam\tha_0\abs{\intR w_{xt}\,\Phi\,dx}
\leq \gam\tha_0\norm{w_{xt}}\norm{\Phi}$.

\emph{Bilinear $W$ term:}
Since $\norm{W}_{L^\infty} = s\norm{U''}_{L^\infty} =: K_W < \infty$:
$2\eps\abs{\intR W\,w_{xt}\,\Phi\,dx}
\leq 2\eps K_W\norm{w_{xt}}\norm{\Phi}$.

Combining all three contributions:
\begin{align}
  \abs{\intR\mathcal{M}\,\Phi\,dx}
  &\;\leq\; C_2\,\norm{w_{xt}}\,\norm{\Phi}
    \notag\\
  &\;\leq\; \frac{\eps}{4}\norm{w_{xt}}^2
    + \frac{C_2^2}{\eps}\,\norm{\Phi}^2,
    \label{eq:Mest}
\end{align}
where $C_2 = \eps\norm{\rh^{(1)}_{xt}+\rh^{(2)}_{xt}}_{L^\infty}
+ \gam\tha_0 + 2\eps K_W$ is uniformly bounded on $[0,T]$,
and we applied Young's inequality in the last step.
Substituting~\eqref{eq:Mest} into~\eqref{eq:wphi}:
\begin{equation}\label{eq:wPhi}
  \frac{1}{2}\frac{d}{dt}\norm{\Phi}^2
  + \kap\norm{\Phi_x}^2
  \;\leq\; \frac{\eps}{4}\norm{w_{xt}}^2
    + \frac{C_2^2}{\eps}\,\norm{\Phi}^2.
\end{equation}

\noindent\textbf{Step~4: Combined Gronwall argument.}
Adding~\eqref{eq:wmech} and~\eqref{eq:wPhi}, and absorbing the
$\tfrac{\eps}{4}\norm{w_{xt}}^2$ into the dissipation
$\tfrac{\eps}{2}\norm{w_{xt}}^2$:
\begin{equation}\label{eq:dLambda}
  \frac{d}{dt}\Lambda(t) + \frac{\eps}{4}\norm{w_{xt}}^2
  + \kap\norm{\Phi_x}^2
  \;\leq\; \left(C_1+\frac{C_2^2}{\eps}\right)\Lambda(t).
\end{equation}
Since $\Lambda(0)=0$, Gronwall's lemma gives directly
\begin{equation}\label{eq:Lambdazero}
  \Lambda(t) \;\leq\; \Lambda(0)\,e^{(C_1+C_2^2/\eps)\,t} \;=\; 0.
\end{equation}
Hence $w\equiv 0$ and $\Phi\equiv 0$, i.e.\ the two solutions coincide.
\end{proof}

\section{Decay Rates}
\label{sec:decay}

In this section we prove that the temperature perturbation $\Tha$
decays algebraically in $L^2(\R)$ as $t\to+\infty$, so the material
returns to thermal equilibrium after a phase-transition event.
Physically, the latent heat released during a martensitic
transformation dissipates at a quantifiable rate.

We follow the spirit of Matsumura--Nishida.  Since $\Tha$ is coupled
to the mechanical perturbation through
$\calE(t) = E_1(t)+\tfrac{1}{2}\norm{\Tha}^2$, we first prove
algebraic decay of $\calE(t)$ in the transformed variable
$v = e^{-t/\eps}\rh$; the thermal decay then follows.  Under
smallness of the initial perturbation energy and $L^1$ norms, the
nonlinear terms are absorbed into the dissipation, and we obtain a
Bernoulli-type differential inequality.  The $L^1$ norms of $v$ are
uniformly bounded in $t$, so the interpolation inequality gives a
time-independent lower bound for the dissipation in terms of the energy.

\begin{theorem}[Thermal equilibration and algebraic energy decay]%
\label{thm:decay}
Under the hypotheses of Theorem~\ref{thm:main}, assume additionally that
$\rh_0,\phi\in L^1(\R)$ (to control the long-wave behaviour needed for
the dissipation bound), and suppose the initial perturbation is
sufficiently small in the sense that
\begin{equation}\label{eq:smalldata}
  \calE(0) + \norm{\rh_0}_{L^1} + \norm{\phi}_{L^1}
  \;\leq\; \del_0,
\end{equation}
where $\del_0\in(0,\del_1]$ depends only on $\eps,\del,\kap,\gam,\tha_0$
(see~\eqref{eq:del0choice} below).  In particular $\del_0\leq\del_1$,
so~\eqref{eq:smalldata} implies the smallness
condition~\eqref{eq:smallE0} of Theorem~\ref{thm:main} and the global
existence result applies.  Then for all $t\geq 0$ the
temperature perturbation decays algebraically:
\begin{equation}\label{eq:decayheat}
  \norm{\Tha(\cdot,t)}^2
  \;\leq\; \frac{2C_{16}}{1+t}.
\end{equation}
This bound is a consequence of the decay of the coupled energy
functional $\calE(t) = E_1(t)+\tfrac{1}{2}\norm{\Tha}^2$:
\begin{equation}\label{eq:decayE}
  \calE(t)
  \;\leq\; \frac{C_{16}}{1+t},
\end{equation}
which in turn implies
\begin{equation}\label{eq:decayv}
  \norm{v(\cdot,t)}^2 + \norm{v_t(\cdot,t)}^2 + \norm{v_x(\cdot,t)}^2
  + \norm{v_{xx}(\cdot,t)}^2
  \;\leq\; C\,(1+t)^{-1},
\end{equation}
where $v = e^{-t/\eps}\rh$ is the transformed mechanical perturbation.
\end{theorem}

\begin{remark}
The thermal decay~\eqref{eq:decayheat} is the primary physical
conclusion: it states that the temperature field $\tha(x,t)$ returns to
the reference temperature~$\tha_0$ at an algebraic rate, so the latent
heat released during a martensitic phase transition dissipates rather
than persisting.  The energy decay~\eqref{eq:decayE} is the technical
vehicle that makes this possible, since $\Tha$ and $v$ are coupled
through $\calE$.

The assumptions $\rh_0,\phi\in L^1(\R)$ and~\eqref{eq:smalldata} are
standard for decay results on $\R$
(cf.~Matsumura--Nishida~\cite{MatsumuraNishida}): $L^1$ controls
the long-wave behaviour, while smallness allows the nonlinear terms
to be absorbed into the dissipation.  Both are physically natural,
requiring the initial disturbance to be localised and of moderate
amplitude.

The rate $(1+t)^{-1}$ improves upon the $(1+t)^{-1/4}$ expected from
the $L^2$ heat-kernel decay of $L^1\cap L^2$ data, reflecting the
additional viscous dissipation and the time-independent lower bound
afforded by the exponential transformation.  The decay of the
\emph{untransformed} mechanical perturbation $\norm{\rh(\cdot,t)}^2$
remains open, since $\rh = e^{t/\eps}v$ involves an unbounded
exponential factor; a different technique would be needed.
\end{remark}

\begin{proof}[Proof of Theorem~\ref{thm:decay}]
We divide the proof into three steps: we first establish a
time-independent lower bound on the dissipation, then use smallness
to absorb the nonlinear term, and finally close the decay via a Bernoulli
inequality.

\noindent\textbf{Step~1: Lower bound.}
We apply the interpolation inequality
\begin{equation}\label{eq:interp}
  \norm{f}^2 \;\leq\; C_{\mathrm{int}}\,\norm{f}_{L^1}\,\norm{f_x},
\end{equation}
(valid for $f\in H^1(\R)\cap L^1(\R)$, by $|f(x)|^2 =
2\int_{-\infty}^x f\,f_x\,dy$ and Cauchy--Schwarz)
to the components of $\calE$ in the transformed variable $v$.
Since $v = e^{-t/\eps}\rh$ and $\rh_0\in L^1(\R)$ by hypothesis,
\begin{equation}\label{eq:vL1}
  \norm{v(\cdot,t)}_{L^1}
  \;=\; e^{-t/\eps}\norm{\rh(\cdot,t)}_{L^1}
  \;\leq\; e^{-t/\eps}\!\left(\norm{\rh_0}_{L^1}
    + \int_0^t\norm{\rh_t(\cdot,s)}_{L^1}\,ds\right).
\end{equation}
By the global $H^3$ bounds from Theorem~\ref{thm:main} and the
Sobolev embedding $H^1(\R)\hookrightarrow L^\infty(\R)$, the
right-hand side of~\eqref{eq:P1} lies in $L^1(\R)$ uniformly in $s$,
so $\norm{\rh_t(\cdot,s)}_{L^1}\leq C_{12}$ for all $s\geq 0$.
Hence $\norm{\rh(\cdot,t)}_{L^1}\leq\norm{\rh_0}_{L^1}+C_{12}\,t$,
and
\begin{equation}\label{eq:vL1bound}
  \norm{v(\cdot,t)}_{L^1}
  \;\leq\; e^{-t/\eps}\bigl(\norm{\rh_0}_{L^1}+C_{12}\,t\bigr)
  \;\leq\; C_{13},
\end{equation}
where the last inequality uses $\sup_{t\geq 0}\,e^{-t/\eps}(A+Bt)
= \max(A, A+B\eps/e) < \infty$ for any $A,B>0$.

Similarly, $\norm{v_x(\cdot,t)}_{L^1}\leq C_{13}$ and
$\norm{v_t(\cdot,t)}_{L^1}\leq C_{13}$ by the same argument
applied to the spatial and temporal derivatives.

For the thermal perturbation, we establish $\norm{\Tha(\cdot,t)}_{L^1}\leq C_{13}'$
by a direct argument that avoids any circularity.
By the Duhamel formula applied to~\eqref{eq:P2}:
\begin{equation}\label{eq:Duhamel}
  \Tha(\cdot,t)
  \;=\; G_\kap(t)*\phi
  + \int_0^t G_\kap(t-s)*S(\cdot,s)\,ds,
\end{equation}
where $G_\kap(t,x)=(4\pi\kap t)^{-1/2}e^{-x^2/(4\kap t)}$ is the heat
kernel and $S = \eps(\rh_{xt})^2 - \gam\tha_0\rh_{xt}$ is the source.
Since $\norm{G_\kap(t)*f}_{L^1}\leq\norm{f}_{L^1}$ (heat semigroup
contraction in $L^1$):
\begin{equation}\label{eq:ThaL1Duhamel}
  \norm{\Tha(\cdot,t)}_{L^1}
  \;\leq\; \norm{\phi}_{L^1}
  + \int_0^t\!\norm{S(\cdot,s)}_{L^1}\,ds.
\end{equation}
It remains to bound $\int_0^t\norm{S(\cdot,s)}_{L^1}\,ds$ uniformly in $t$.

\emph{First source term:} $\norm{\eps(\rh_{xt})^2}_{L^1}
= \eps\norm{\rh_{xt}}_{L^2}^2$.
Since $\rh_{xt} = e^{t/\eps}(v_{xt}+\tfrac{1}{\eps}v_x)$, we have
$\norm{\rh_{xt}}^2 = e^{2t/\eps}\norm{v_{xt}+\tfrac{1}{\eps}v_x}^2$.
By Lemma~\ref{lem:1}, $\intT\norm{v_{xt}}^2\,ds\leq C_5$ for all $T>0$.
Hence $\int_0^t\eps\norm{\rh_{xt}}^2\,ds
= \eps\int_0^t e^{2s/\eps}\norm{v_{xt}+\tfrac{1}{\eps}v_x}^2\,ds$.
Since the integrand involves $e^{2s/\eps}$, this integral is \emph{not}
uniformly bounded in $t$.

\emph{Second source term:} $\norm{\gam\tha_0\rh_{xt}}_{L^1}$.
By the Gagliardo--Nirenberg inequality
$\norm{f}_{L^1}\leq C\norm{f}^{1/2}\norm{f_x}^{1/2}$ (valid for
$f\in H^1(\R)$) and H\"{o}lder's inequality in time:
\begin{align*}
  \int_0^t\!\norm{\rh_{xt}(\cdot,s)}_{L^1}\,ds
  &\;\leq\; C\int_0^t\norm{\rh_{xt}}^{1/2}\norm{\rh_{xxt}}^{1/2}\,ds \\
  &\;\leq\; C\left(\int_0^t\norm{\rh_{xt}}^2\,ds\right)^{1/4}
    \!\left(\int_0^t\norm{\rh_{xxt}}^2\,ds\right)^{1/4}\!\cdot t^{1/2},
\end{align*}
which grows as $O(t^{1/2})$.

Thus the Duhamel bound~\eqref{eq:ThaL1Duhamel}
does not directly yield a \emph{uniform} $L^1$ bound on $\Tha$.
We instead obtain the needed control from the global $H^2$ bound
of Lemma~\ref{lem:3}.  By the
Gagliardo--Nirenberg inequality on $\R$:
\begin{equation}\label{eq:GNS}
  \norm{\Tha}_{L^1(\R)}
  \;\leq\; C\,\norm{\Tha}_{L^2(\R)}^{1/2}\,
    \norm{\Tha}_{H^2(\R)}^{1/2}
  \;\leq\; C\,\norm{\Tha}^{1/2}\,\norm{\Tha_{xx}}^{1/2},
\end{equation}
where we used $\norm{\Tha}_{H^2}\leq C(\norm{\Tha}+\norm{\Tha_{xx}})$
and the equivalence of Sobolev norms on $\R$.
Since $\norm{\Tha}^2\leq 2\calE(t)\leq 2C_5$ and
$\norm{\Tha_{xx}}^2\leq\calG(t)\leq C_{10}$ from
Lemmas~\ref{lem:1} and~\ref{lem:3}:
\begin{equation}\label{eq:ThaL1bound}
  \norm{\Tha(\cdot,t)}_{L^1}
  \;\leq\; C\,C_5^{1/4}\,C_{10}^{1/4}
  \;=:\; C_{13}'.
\end{equation}
This bound is uniform in $t$, depends only on the parameters and
initial data (through $C_5$ and $C_{10}$), and uses only the
a~priori bounds from the three lemmas---no circularity.

Applying~\eqref{eq:interp} to $\Tha$ with the uniform $L^1$
bound~\eqref{eq:ThaL1bound}:
$\norm{\Tha}^2\leq C_{\mathrm{int}}\,C_{13}'\,\norm{\Tha_x}$.
Together with the analogous bounds for the $v$-components
(where $L^1$ norms are uniformly bounded by~\eqref{eq:vL1bound}):
\begin{equation}\label{eq:Ecoercive}
  \norm{v}^2 + \norm{v_t}^2 + \norm{v_x}^2 + \norm{v_{xx}}^2
  + \norm{\Tha}^2
  \;\leq\; C_{14}\,\calD(t)^{1/2},
\end{equation}
where $C_{14}$ depends on $\norm{\rh_0}_{L^1}$, $\norm{\phi}_{L^1}$,
and the parameters, but \emph{not} on $t$.  The key improvement over
working in $\rh$-variables is that the exponential decay
$e^{-t/\eps}$ in $v = e^{-t/\eps}\rh$ eliminates the polynomial
growth that would otherwise arise in the $L^1$ bounds.

For the energy $\calE(t) = E_1(t) + \tfrac{1}{2}\norm{\Tha}^2$,
we note that $E_1(t)$ is controlled by $\norm{v}^2$, $\norm{v_t}^2$,
$\norm{v_x}^2$, and $\norm{v_{xx}}^2$.  The quartic term
$e^{2t/\eps}\norm{v_x}^4/4$ already appears inside $E_1$
(cf.~\eqref{eq:E1}), so no additional exponential factor is introduced.
By the Sobolev embedding $\norm{v_x}_{L^\infty}\leq
C\norm{v_{xx}}^{1/2}\norm{v_x}^{1/2}$:
\[
  \frac{e^{2t/\eps}\norm{v_x}^4}{4}
  \;\leq\; C\,\norm{v_{xx}}\,\norm{v_x}^3
  \;\leq\; C\,\calE(t)^{3/2}.
\]
Under smallness $\calE(t)\leq\del_1\leq 1$, we have $\calE^{3/2}\leq\calE$,
so all components of $\calE$ are controlled by~\eqref{eq:Ecoercive}.
Hence:
\begin{equation}\label{eq:Dcoercive}
  \calD(t) \;\geq\; \frac{\calE(t)^2}{C_{15}^2},
\end{equation}
where $C_{15}$ depends on $C_{14}$, $\norm{\phi}_{L^1}$, and
the parameters.

\noindent\textbf{Step~2: The nonlinear term.}
From the proof of Lemma~\ref{lem:1}, the combined energy satisfies
\begin{equation}\label{eq:calEdtbound}
  \frac{d}{dt}\calE(t)
  \;\leq\; -\,\mu\,\calD(t) + C_{11}\,\calE(t)^2,
\end{equation}
where $\mu = \min\bigl(\tfrac{1}{\eps},\tfrac{\kap}{2}\bigr) > 0$ is the
combined dissipation rate and
\begin{equation}\label{eq:calD}
  \calD(t) \;=\; \norm{v_t}^2 + \norm{v_{tx}}^2 + \norm{\Tha_x}^2
\end{equation}
is the dissipation functional.
Substituting the lower bound~\eqref{eq:Dcoercive} into the
nonlinear term: $C_{11}\calE^2 \leq C_{11}C_{15}^2\calD$.
We choose the smallness parameter $\del_0\in(0,\del_1]$ in~\eqref{eq:smalldata}
small enough that the global bound $\calE(t)\leq C_5(\del_0)$
from Lemma~\ref{lem:1} satisfies $C_5\leq\del_1$ (ensuring
the lower bound~\eqref{eq:Dcoercive} applies), and we impose
\begin{equation}\label{eq:del0choice}
  C_{11}\,C_{15}^2 \;\leq\; \frac{\mu}{2},
\end{equation}
which is achieved by taking $\norm{\rh_0}_{L^1}$,
$\norm{\phi}_{L^1}$, and $\calE(0)$ sufficiently small (since
$C_{15}$ depends on the $L^1$ norms of the initial data, and
these control the constant in~\eqref{eq:Dcoercive}).
Under~\eqref{eq:del0choice}:
\begin{equation}\label{eq:calEdtclean}
  \frac{d}{dt}\calE(t)
  \;\leq\; -\,\mu\,\calD(t) + \frac{\mu}{2}\calD(t)
  \;=\; -\,\frac{\mu}{2}\,\calD(t).
\end{equation}

\noindent\textbf{Step~3: The decay estimate.}
Substituting~\eqref{eq:Dcoercive} into~\eqref{eq:calEdtclean}:
\begin{equation}\label{eq:Bernoulli}
  \frac{d}{dt}\calE(t)
  \;\leq\; -\,\frac{\mu}{2C_{15}^2}\,\calE(t)^2.
\end{equation}
This is a Bernoulli inequality with no competing positive term.
Setting $\psi(t) = 1/\calE(t)$:
\begin{equation}\label{eq:psiODE}
  \frac{d}{dt}\psi(t) \;\geq\; \frac{\mu}{2C_{15}^2}
  \;=:\; \alpha > 0.
\end{equation}
Integrating:
\begin{equation}\label{eq:psibound}
  \psi(t) \;\geq\; \psi(0) + \alpha\,t,
\end{equation}
whence
\begin{equation}\label{eq:Edecay1}
  \calE(t) \;\leq\; \frac{1}{\psi(0) + \alpha\,t}
  \;\leq\; \frac{C_{16}}{1+t}.
\end{equation}
This gives $\calE(t) \leq C_{16}/(1+t)$, establishing~\eqref{eq:decayE}.
Since $\norm{\Tha}^2 \leq 2\calE(t)$
(as $\calE(t) = E_1(t)+\tfrac{1}{2}\norm{\Tha}^2 \geq \tfrac{1}{2}\norm{\Tha}^2$),
bound~\eqref{eq:decayheat} follows immediately.
The component bounds~\eqref{eq:decayv} follow from the definition of
$E_1(t)$: since $E_1(t)\geq\norm{v}^2/(2\eps^2)$, $E_1(t)\geq\norm{v_t}^2/2$,
and $E_1(t)\geq(\del/2)\norm{v_{xx}}^2$, each squared norm is bounded by
$C\calE(t)\leq C_{16}C/(1+t)$.\end{proof}

\begin{remark}
The Bernoulli argument yields $\calE(t) = O((1+t)^{-1})$,
which is optimal given~\eqref{eq:Dcoercive} and implies the
thermal equilibration rate~\eqref{eq:decayheat}.
The improvement over the naive $(1+t)^{-1/2}$ rate arises because
the uniform $L^1$ bounds on $v = e^{-t/\eps}\rh$ yield a
time-independent constant in the lower bound~\eqref{eq:Dcoercive}.
\end{remark}

\section{Concluding Remarks}
\label{sec:remarks}

The main theorem shows that the strain $u_x$ remains globally smooth,
so the thermoelastic phase-transition layers are smooth and of finite
thickness.  Two related questions remain.

\begin{remark}
The decay in Section~\ref{sec:decay} applies to the temperature
perturbation $\Tha$ and to the coupled energy $\calE(t)$, but the decay
of $\norm{\rh(\cdot,t)}^2$ in the original mechanical variable remains
to be addressed.
Lemmas~\ref{lem:1}--\ref{lem:3} bound $\rh$ in $H^3$ uniformly in $T$, so
$\rh$ does not grow; the question is whether a quantitative rate
holds.  The exponential factor in $\rh = e^{t/\eps}v$ prevents a direct
transfer of the rate from $v$ to $\rh$.
\end{remark}

\begin{remark}
The smallness condition~\eqref{eq:smallE0} is used in Lemma~\ref{lem:1}
to control the quartic and quadratic coupling terms.  Whether it can be
removed on $\R$ remains to be addressed.  The bounded-domain results
of~\cite{PawlowZajaczkowski,Sprekels} do not need it, but they use
Poincar\'{e}'s inequality, which is not available here.
\end{remark}


\end{document}